\documentclass[12pt]{article}%
\usepackage{graphicx}
\usepackage{amsmath}
\usepackage{amsfonts}
\usepackage{amssymb}%
\providecommand{\U}[1]{\protect\rule{.1in}{.1in}}
\newtheorem{theorem}{Theorem}
\newtheorem{acknowledgement}[theorem]{Acknowledgement}

\newtheorem{conjecture}[theorem]{Conjecture}

\newtheorem{definition}[theorem]{Definition}

\newenvironment{proof}[1][Proof]{\noindent\textbf{#1.} }{\ \rule{0.5em}{0.5em}}
\begin{document}

\title{\textbf{Topological Tverberg Theorem: }\\\textbf{the proofs and the counterexamples\footnotetext{Part of this work has
been carried out in the framework of the Labex Archimede (ANR-11-LABX-0033)
and of the A*MIDEX project (ANR-11-IDEX-0001-02), funded by the
\textquotedblleft Investissements d'Avenir" French Government programme
managed by the French National Research Agency (ANR). Part of this work has
been carried out at IITP RAS. The support of Russian Foundation for Sciences
(project No. 14-50-00150) is gratefully acknowledged.}\thanks{ This paper
clarifies some bibliographical issues in the review paper \cite{S} in this
volume, which clarification I was entrusted to make by the Editorial Board of
Russian Math. Surveys.}}}
\author{Senya Shlosman\\Skolkovo Institute of Science and Technology, \\Moscow, Russia;\\Aix Marseille University, University of Toulon, \\CNRS, CPT, Marseille, France;\\Inst. of the Information Transmission Problems,\\RAS, Moscow, Russia.\\S.Shlosman@skoltech.ru, senya.shlosman@univ-amu.fr, \\shlos@iitp.ru}
\maketitle

\begin{abstract}
I describe the history of Topological Tverberg Theorem (TTT). I present some
important constructions and discuss their properties. In particular, I
describe in details the cell structure of the classifying space $K\left(
S_{r},1\right)  ,$ where $S_{r}$ is the permutation group.

\end{abstract}

\section{My initiation to TTT}

My introduction to TTT took place in the Spring of 1980, in the student menza
of MSU (Zone B). During lunch there, Imre Barany was telling me about his
success with the topological version of Radon Theorem (TRT). The
two-dimensional (for simplicity) version of TRT claims the following. Let
$\Delta_{3}$ be a 3D simplex. A pair of faces $\Delta^{\prime},\Delta
^{\prime\prime}\subset\partial\Delta_{3}$ is called complimentary, if
$\Delta^{\prime}\cap\Delta^{\prime\prime}=\varnothing,$ and $\dim\left(
\Delta^{\prime}\right)  +\dim\left(  \Delta^{\prime\prime}\right)  =2.$ For
example, $\Delta^{\prime}$ can be a 2D face and $\Delta^{\prime\prime}$ -- the
opposite vertex ($\equiv$ 0D simplex), the other possibility is that
$\Delta^{\prime},\Delta^{\prime\prime}$ is a pair of skew edges.

\begin{theorem}
TRT Let $f:\partial\Delta_{3}\rightarrow R^{2}$ be a continuous map. Then for
some pair $\Delta^{\prime},\Delta^{\prime\prime}$ of complimentary faces of
$\Delta_{3}$ one has
\[
f\left(  \Delta^{\prime}\right)  \cap f\left(  \Delta^{\prime\prime}\right)
\neq\varnothing.
\]

\end{theorem}

My immediate reaction to this statement was that it should follow from the
Borsuk--Ulam Theorem (BUT) about antipodal points: BUT claims that if
$g:\mathbb{S}^{2}\rightarrow\mathbb{R}^{2}$ is a continuous map, then for some
pair $x^{\prime},x^{\prime\prime}\in\mathbb{S}^{2}$ of antipodal points we
have $g\left(  x^{\prime}\right)  =g\left(  x^{\prime\prime}\right)  .$
Indeed, the two statements look similar; the only step needed, in order to
relate the two, is to find a `universal' map $F:\mathbb{S}^{2}\rightarrow
\partial\Delta_{3}$ such that the images of each pair of antipodal points
belong to the complimentary faces of $\partial\Delta_{3}.$

To find such a map is not hard, and here it is. First, let us replace the
sphere by the cuboctahedron, $\mathcal{C}$:

%\vspace{-1.7cm}
%\begin{figure}[h]
\begin{figure}[h]
\centering
\includegraphics[scale=0.4]{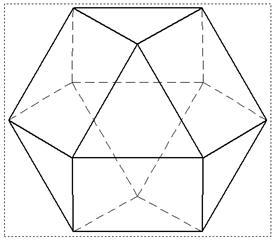}\end{figure}
%\vspace{-1.8cm}
Note that $\mathcal{C}$ has 8 triangles $\delta_{1},...,\delta_{8}$ and 6
squares, i.e. 14 faces in total, while $\partial\Delta_{3}$ has 4 triangles
$d_{1},...,d_{4},$ 4 vertices $v_{1},...,v_{4},$ and 6 edges, which makes
again 14 items. So let us select 4 triangles out of 8 in $\mathcal{C}$ in such
a way that each two are disjoint -- say, $\delta_{1}$ to $\delta_{4}$ -- and
define $F$ on these triangles to be linear isometries to, resp.,
$d_{1},...,d_{4}.$ Supposing that the faces $\delta_{5},...,\delta_{8}$ are
opposite to $\delta_{1},...,\delta_{4}$ and vertices $v_{1},...,v_{4}$
disjoint with faces $d_{1},...,d_{4},$ we define $F\left(  \delta_{k}\right)
\equiv v_{k-4}$ for $k=5,...,8.$ It remains to continue $F$ to the six
squares. But $F$ is defined already on their boundaries, sending each boundary
onto a corresponding edge. So we continue $F$ inside each square in such a way
that it takes the whole square into that same edge.

The existence of $F$ thus proves TRT (in 2D).

Imre also told me about his project to generalize the Linear Tvebrerg Theorem
(LTT) to Topological Tvebrerg Theorem, and we decided to join our efforts.
Shortly, Andras Szucs joined our team. Our goal was to prove the

\begin{conjecture}
\textbf{TTT}: Let $f:\partial\Delta_{(d+1)(r-1)}\rightarrow\mathbb{R}^{d}$ be
continuous. Then for some collection $\Delta^{1},...,\Delta^{r}\subset
\partial\Delta_{(d+1)(r-1)}$ of complimentary subsimplices we have
\[
f\left(  \Delta^{1}\right)  \cap...\cap f\left(  \Delta^{r}\right)
\neq\varnothing.
\]

\end{conjecture}

(LTT is the same statement, but for $f$ linear. It was proven by Tverberg in
\cite{T}.)

Our plan was somewhat similar to the above. For a simplex $\Delta_{N}$ we have
introduced the complex $Y_{N,r}$ of $r$-tuples $\left(  y_{1},...,y_{r}%
\right)  \subset\Delta_{N}$, where the points $y_{1},...,y_{r}$ lie in
disjoint faces of $\Delta_{N}.$ $Y_{N,r}$\textbf{ }was going to play the role
of cuboctahedron\textbf{ }$\mathcal{C}.$ That turned out to be a very nice CW
complex, see below. It has an evident free action of the cyclic group
$\mathbb{Z}_{r}.$ We were using the cyclic action $\mathbb{Z}_{r}$ on
$\mathbb{R}^{nr},$ which sends $\left(  w_{1},...,w_{r}\right)  $ to $\left(
w_{2},,...,w_{r},w_{1}\right)  $, $w_{i}\in\mathbb{R}^{n}.$ We needed it to be
free action away from the diagonal $\mathbb{R}^{n}\subset\mathbb{R}^{nr},$
comprised by the vectors $\left(  w,...,w\right)  \in\mathbb{R}^{nr},$
$w\in\mathbb{R}^{n}$ (the antipodal map above corresponds to $\mathbb{R}%
^{2}\subset\mathbb{R}^{4}$). That, unfortunately, happens only when $r$ is
prime. That restricted our proof of TTT in \cite{BSS} to the case of prime
$r$-s only, but I thought at that time that the primality obstacle will be of
restricted temporal nature.

\section{The CW complex $Y_{N,r}$}

The complex $Y_{N,r}$ has several nice properties. First, not only the cyclic
group $\mathbb{Z}_{r},$ but the whole permutation group $S_{r}$ acts freely on
it. Next, all the homotopy groups $\pi_{k}\left(  Y_{N,r}\right)  =0$ for
$k\leq N-r,$ see \cite{BSS}. The first nontrivial homotopy group is
$\pi_{N-r+1}\left(  Y_{N,r}\right)  ,$ while the dimension $\dim\left(
Y_{N,r}\right)  $ is also $N-r+1.$ That means that $Y_{N,r}$ is homotopy
equivalent to the $\left(  N-r+1\right)  $-dimensional skeleton of the
universal cover of the classifying complex $K\left(  S_{r},1\right)  .$ In
other words, $\left[  K\left(  S_{r},1\right)  \right]  _{N-r+1}=Y_{N,r}%
/S_{r}.$ This is the most economical model of the skeleton of $K\left(
S_{r},1\right)  .$

For example, for $r=2$ we have $K\left(  S_{r},1\right)  =K\left(
\mathbb{Z}_{2},1\right)  =\mathbb{RP}^{\infty},$\textbf{ }and so the complex
$Y_{N,2}$ has a homotopy type of the sphere\textbf{ }$\mathbb{S}^{N-1}%
.$\textbf{ }For $N=3$ we get\textbf{ }$\mathbb{S}^{2},$\textbf{ }with the cell
structure of the cuboctahedron\textbf{ }$\mathcal{C}.$ When $N=2$ the
complex\textbf{ }$Y_{2,2}\sim\mathbb{S}^{1}$ is a hexagon, with edges $\left(
0\right)  \left(  1,2\right)  ,$ $\left(  1\right)  \left(  2,0\right)  ,$
$\left(  2\right)  \left(  0,1\right)  ,$ $\left(  0,1\right)  \left(
2\right)  ,$ $\left(  1,2\right)  \left(  0\right)  $ and $\left(  2,0\right)
\left(  1\right)  .$

The CW complexes $Y_{N,r}$ come with some extra structure. Namely, every cell
$c\subset Y_{N,r}$ is equipped with the structure of a product, $c=\Delta
_{d_{1}}\times...\times\Delta_{d_{r}},$ where $\Delta_{d}$ denotes the
$d$-dimensional simplex, $d\geq0,$ with $d_{1}+...+d_{r}=\dim c.$ (In
particular, every sphere gets such a cell structure, each cell having this
product structure.) LWe will call such a CW complex a \textit{prism }complex.

Together with Oleg Ogievetsky we spent many happy hours attempting to prove
TTT for all values of $r.$ That was a very exiting experience; we saw many
beautiful questions related to TTT and we invented several nice constructions.
I will present one such invention: the concept of $\mathcal{O}$-orientation.

Before explaining it I will talk about the classic orientability. One way of
saying the prism complex $Y$ to be orientable is to demand that

\begin{itemize}
\item every cell $c$ of dimension $\dim Y-1$ belongs to the boundaries of
exactly two cells $c^{\prime},c^{\prime\prime}$ of dimension $\dim Y;$

\item every cell of dimension $\dim Y$ can be assign an orientation in such a
way that each cell $c$ of dimension $\dim Y-1$ inherits the opposite
orientations from its two parent cells $c^{\prime},c^{\prime\prime}.$
\end{itemize}

For example, the complexes $Y_{N,2}=\mathbb{S}^{N-1}$ are orientable.

The complexes $Y_{N,r}$ are not orientable in the above sense once $r>2.$ The
reason is simple: its faces of codimension 1 belong to $r>2$ faces of full
dimension. Yet, for our experiments we needed some canonical way of choosing
the orientations of its $\left(  N-r+1\right)  $-dimensional faces of top
dimension. The natural way of doing it is to use the $\mathcal{O}%
$-\textit{orientability} of $Y_{N,r}$:

\begin{definition}
Let $Y$ be a CW prism complex of dimension $d.$ It is called $\mathcal{O}%
$-orientable, if there is a choice of orientations of all its $d$-cells, such
that each cell of $Y$ of codimension $1$ inherits the \textbf{same}
orientation from all $d$-dimensional cells to which it is incident.
\end{definition}

Unlike orientability, the $\mathcal{O}$-orientability depends not only on the
homotopy type of $Y,$ but also on the prism cell structure. Some
triangulations of the sphere $\mathbb{S}^{2}$ are not $\mathcal{O}%
$-orientable. On the other hand, some triangulations of the Mobius strip are
$\mathcal{O}$-orientable.

\begin{theorem}
\cite{OS} The complex $Y_{N,r},$ with the prism structure described, is
$\mathcal{O}$-orientable.
\end{theorem}

\begin{proof}
Let us consider a face (=cell) $F=F\left(  V_{1},...,V_{r}\right)  $ of
$Y^{N,r}$ having full dimension. It corresponds to the ordered partition
$V_{1},...,V_{r}$ of the vertex set $\left\{  v_{0},v_{1},...v_{j}%
,...,v_{N}\right\}  $ of the simplex $\Delta_{N}$ into $r$ disjoint non-empty
subsets. In that case $F$ is the product of simplices, $F=\Delta\left(
V_{1}\right)  \times...\times\Delta\left(  V_{r}\right)  ,$ so to fix the
orientation of $F$ is the same as to fix the orientations on each simplex
$\Delta\left(  V_{i}\right)  $ and to take them in the order given. In case
some $V$ is a singleton, this orientation turns into a sign, $\pm1.$ The
orientation of the simplex $\Delta\left(  v_{i_{0}},...,v_{i_{k}}\right)  $ is
specified by fixing some total order $\overrightarrow{v_{i_{0}},...,v_{i_{k}}%
}$ on the set $v_{i_{0}},...,v_{i_{k}}$ of its vertices. For a
zero-dimensional oriented simplex corresponding to a vertex $v$ of $\Delta
_{N}$ we will use the notation $\Delta\left(  \overset{\pm}{v}\right)  .$

Likewise, a face $G$ of codimension 1 of $Y^{N,r},$ which belongs to the
boundary of the face $F,$ corresponds to the ordered partition $V_{1}^{\prime
},...,V_{r}^{\prime}$ of the vertex set $\left\{  v_{0},v_{1},...,\widehat
{v_{j}}^{j},...,v_{N}\right\}  $ into $r$ disjoint non-empty subsets, where
the operation $\widehat{\mathbf{\ast}}^{j}$ means the removal of the $j$-th
term, $j=0,1,...$ . So all the elements of the partition $V_{1}^{\prime
},...,V_{r}^{\prime}$ except one, say, $V_{k}^{\prime},$ coincide with the
subsets $V_{1},...,V_{r},$ while $V_{k}^{\prime}=V_{k}\setminus v^{G},$
$v^{G}\in\left\{  v_{0},v_{1},...v_{j},...,v_{N}\right\}  .$ So,
$G=\Delta\left(  V_{1}\right)  \times...\times\Delta\left(  V_{k-1}\right)
\times\Delta\left(  V_{k}^{\prime}\right)  \times...\Delta\left(
V_{r}\right)  ,$ and $\Delta\left(  V_{k}^{\prime}\right)  \in\partial
\Delta\left(  V_{k}\right)  .$

Now we remind the reader the standard definition of the inherited orientation.
The orientation of the simplex $\Delta\left(  v_{i_{0}},...,v_{i_{s}}\right)
$ induces the orientations of all its faces according to the formula
\begin{equation}
\partial\Delta\left(  V_{k}\right)  \equiv\partial\Delta\left(  v_{i_{0}%
},...,v_{i_{s}}\right)  =\sum_{j=0}^{s}\left(  -1\right)  ^{j}\Delta\left(
v_{i_{0}},v_{i_{1}},...,\widehat{v_{i_{j}}}^{j},...,v_{i_{s}}\right)  .
\label{10}%
\end{equation}
The meaning is that if the vertex removed $v^{G}$ happens to be at the odd
position in the ordered string $\overrightarrow{v_{i_{0}},...,v_{i_{s}}}$
(i.e. $j=0,2,4,...$) defining the (oriented) simplex $\Delta\left(
V_{k}\right)  ,$ then it has just to be removed from it, keeping the order of
the remaining vertices as defining the orientation of the simplex
$\Delta\left(  V_{k}^{\prime}\right)  $; otherwise, the removal of $v^{G}$ has
to be supplemented by the transposition of, say, the first two of the
remaining vertices defining $\Delta\left(  V_{k}^{\prime}\right)  .$ If
$V_{k}^{\prime}$ happens to be a singleton, then the orientation of
$\Delta\left(  V_{k}^{\prime}\right)  $ is the sign it gets in $\left(
\ref{10}\right)  .$ In case of a product of simplices one has to use the
Leibniz rule:%
\begin{align*}
&  \partial\left[  \Delta\left(  V_{1}\right)  \times...\times\Delta\left(
V_{r}\right)  \right] \\
&  =\sum_{k=1}^{r}\left(  -1\right)  ^{\dim\left(  \Delta\left(  V_{1}\right)
\times...\times\Delta\left(  V_{k-1}\right)  \right)  }\Delta\left(
V_{1}\right)  \times...\times\Delta\left(  V_{k-1}\right)  \times
\partial\Delta\left(  V_{k}\right)  \times...\times\Delta\left(  V_{r}\right)
.
\end{align*}

Now we are going to make the choice of orientation for every face $F.$ First,
let us take the face $F_{0}=F\left(  \left\{  v_{0}\right\}  ,\left\{
v_{1}\right\}  ,...,\left\{  v_{r-2}\right\}  ,\left\{  v_{r-1},...,v_{N}%
\right\}  \right)  ,$ which corresponds to the partition into $r-1$ singletons
and the remaining set of $N-r+2$ points; this face is $\left(  N-r+1\right)
$-dimensional simplex. We choose its orientation according to the order
$v_{r-1}<...<v_{N},$ while all singletons get the $+1$ orientation:
\[
F_{0}=F\left(  \left\{  \overset{+}{v_{0}}\right\}  ,\left\{  \overset
{+}{v_{1}}\right\}  ,...,\left\{  \overset{+}{v_{r-2}}\right\}  ,\left\{
\overrightarrow{v_{r-1},...,v_{N}}\right\}  \right)  .
\]

Let now $F=F\left(  \left\{  \overrightarrow{v_{1,1},...,v_{1,i_{1}}}\right\}
,\left\{  \overrightarrow{v_{2,1},...,v_{2,i_{2}}}\right\}  ,...,\left\{
\overrightarrow{v_{r,1},...,v_{r,i_{r}}}\right\}  \right)  $ be an arbitrary
face, with orientation corresponding to the orders $v_{1,1}<...<v_{1,i_{1}},$
$v_{2,1}<...<v_{2,i_{2}},$ $...,$ $v_{r,1}<...<v_{r,i_{r}}$ on its factors,
and $\left(  +\right)  $\textit{-signs of singletons}. We will formulate the
condition this orientation of $F$ has to satisfy, in order to define the
global $\mathcal{O}$-orientation of $Y^{N,r}$. Let us correspond to this
orientation the string $S\left(  F\right)  $ of $N+r$\textbf{ }symbols, which
are letters $\left\{  v_{0},...,v_{N}\right\}  \cup\left\{  s_{1}%
,...,s_{r-1}\right\}  .$ The string $S\left(  F\right)  $ is the following:%
\[
S\left(  F\right)  =v_{1,1},...,v_{1,i_{1}},s_{1},v_{2,1},...,v_{2,i_{2}%
},s_{2},...,s_{r-1},v_{r,1},...,v_{r,i_{r}}.
\]
The only condition the orientation on $F$ has to satisfy, is that the parity
of the permutation $\pi\left(  F\right)  \in S_{N+r}$, which takes $S\left(
F\right)  $ into the string%
\[
S\left(  F_{0}\right)  =v_{0},s_{1},v_{1},s_{2},...,s_{r-1},v_{r-1}%
,v_{r},...,v_{N},
\]
is even.

Let us check that the orientations of faces thus prescribed, form an
$\mathcal{O}$-orientation.

Let us start with the cuboctahedron $\mathcal{C}.$ Of course, $\mathcal{C}$ is
orientable, but its orientation is not an $\mathcal{O}$-orientation. Its
$\mathcal{O}$-orientation is easy to guess: it is comprised by orienting all
squares one way, and all triangles -- the opposite way. Let us check that the
$\mathcal{O}$-recipe above gives the same result.

We are talking about the 2D complex $Y^{3,2},$ built from the simplex
$\Delta_{3}=\left\{  0,1,2,3\right\}  .$ Let us take the (triangular) face
$F_{0}=\left(  \left\{  0\right\}  ,\left\{  1,2,3\right\}  \right)  $ and one
of its boundary edges, say, the segment $G=\left(  \left\{  0\right\}
,\left\{  1,3\right\}  \right)  .$ The vertex $v^{G}$ is the point $\left\{
2\right\}  ,$ missing in $G$. The segment $G$ is adjacent to another 2D face,
the square face $F_{1}=\left(  \left\{  0,2\right\}  ,\left\{  1,3\right\}
\right)  .$

The orientation of $F_{0}$ is the $\left(  +1\right)  $ orientation of the
point $\left\{  0\right\}  $ and the orientation of the simplex $\Delta\left(
1,2,3\right)  ,$ corresponding to the order $1<2<3,$ $F_{0}=\left(  \left\{
\overset{+}{0}\right\}  ,\left\{  \overrightarrow{1,2,3}\right\}  \right)  .$
Let us check which orientation the face $F_{1}$ gets via our recipe above:
$\left(  \left\{  \overrightarrow{0,2}\right\}  ,\left\{  \overrightarrow
{1,3}\right\}  \right)  $ or $\left(  \left\{  \overrightarrow{0,2}\right\}
,\left\{  \overrightarrow{3,1}\right\}  \right)  .$ To find it, let us look at
the parity of the, say, first permutation:%
\[
\left(
\begin{array}
[c]{ccccc}%
0 & s_{1} & 1 & 2 & 3\\
0 & 2 & s_{1} & 1 & 3
\end{array}
\right)  .
\]
Since it is even, we take for the $\mathcal{O}$-orientation of the face
$F_{1}$ the one defined by the ordering $\left(  \left\{  \overrightarrow
{0,2}\right\}  ,\left\{  \overrightarrow{1,3}\right\}  \right)  .$

The orientation that the edge $G=\left(  \left\{  0\right\}  ,\left\{
1,3\right\}  \right)  $ inherits from the oriented face $F_{0}=\left(
\left\{  \overset{+}{0}\right\}  ,\left\{  \overrightarrow{1,2,3}\right\}
\right)  ,$ corresponds to the ordering $G=\left(  \left\{  \overset{+}%
{0}\right\}  ,\left\{  \overrightarrow{3,1}\right\}  \right)  ,$ because the
index $j$ of the formula $\left(  \ref{10}\right)  $ has value $1;$ the
corresponding term there is $\left(  -1\right)  \Delta\left(  1,\widehat
{2},3\right)  $. The orientation the edge $G=\left(  \left\{  0\right\}
,\left\{  1,3\right\}  \right)  $ inherits from the oriented face
$F_{1}=\left(  \left\{  \overrightarrow{0,2}\right\}  ,\left\{
\overrightarrow{1,3}\right\}  \right)  ,$ corresponds to the ordering
$G=\left(  \left\{  \overset{-}{0}\right\}  ,\left\{  \overrightarrow
{1,3}\right\}  \right)  ,$ with the $-1$ sign for the vertex $\left\{
0\right\}  $ since $\partial\Delta\left(  0,2\right)  =+\left(  2\right)
-\left(  0\right)  .$ So the two orientations of $G$ do agree. The check for
other edges is done in the same way.

The general case is obtained by induction.
\end{proof}

\section{Beyond the prime case}

The start of the next chapter of TTT came with the proof of TTT for $r$ being
prime power. It was done in \cite{O} -- a famous paper left unpublished. So I
have learned about this fact much later, from an independent paper by
Volovikov, \cite{V}.

Thus, the general case of TTT seemed to be within reach. Yet, nobody was able
to make the final step. Some people even expressed doubts. The one case I know
was in 2011, when I was discussing TTT with David Kazhdan. But I was not
convinced. Bad luck! -- because four years later Florian Frick came with
counterexamples to TTT, in \cite{F}. He has shown that for any $r$ not a prime
power the TTT does not hold...

That was the end of one chapter of TTT. Meanwhile, many other open questions
around TTT have appeared, see e.g. \cite{BBZ} -- but this is another story.

\section{Credits}

In the report \cite{S} one reads:

\textit{For the counterexample, papers ... by M. Ozaydin, M. Gromov, P.
Blagojevic, F. Frick, G. Ziegler, I. Mabillard and U. Wagner are important.}

This is a correct statement, since the paper \cite{F} appears in this list.
But this statement is an understatement; it is the truth, but not the whole
truth: -- without \cite{F} there would be no counterexample. In general, it
never happens that a proof of a theorem is contained in several papers by
several authors. Usually, there exists a pivotal paper, such that the proof in
question does not exists before it, and does exist after it. In the case of
counterexample to TTT such a paper was written by Florian Frick, building on
earlier results of Mabillard and Wagner.

The result of Frick is based on the constraint method, developed in 2015 in
\cite{BFZ}. Again, the paper \cite{S} credits M. Gromov for the discovery of
this method earlier, in his 2010 paper \cite{G}. This monumental paper of
Mikhail Leonidovich Gromov is indeed an outstanding work. Yet, crediting this
(110-pages-)paper for the constraint method (referring to a sketch in the
discussion section there) is rather detrimental to this paper than otherwise.
This half-page sketch of a proof is incomplete and contains typos, so much so
that an attempt of correcting them is made in \cite{S}. Out of respect to
Mikhail Gromov it would have been better not to initiate this discussion.

The mathematical part of A. Skopenkov review is written nicely, so I recommend
it to the readership of UMN for a clear introduction to TTT.

\begin{acknowledgement}
1. I am grateful to Oleg Ogievetsky for his permission to reproduce our joint
results on $\mathcal{O}$-orientation in this paper.

2. The results on $\mathcal{O}$-orientation were obtained in CPT, the other
results on the topology of the complex $Y^{N,r}$ -- in IITP. I thank both
institutions for the stimulating working environment.
\end{acknowledgement}

\end{document}